\documentclass[11pt]{article}
\usepackage{mathrsfs}
\usepackage{tikz}
\usepackage{latexsym,lineno}
\usepackage{epsfig}
\usepackage{color}
\usepackage{amsmath}\usepackage{fleqn}\usepackage{verbatim}\usepackage{epsf}
\usepackage{amsthm}\usepackage{graphicx, float}\usepackage{graphicx}
\usepackage{amsfonts}\usepackage{amssymb}\usepackage{graphpap}
\usepackage{epic}\usepackage{curves}
 
\newcommand{\be}{\begin{equation}}
\newcommand{\ee}{\end{equation}}
\newcommand{\benum}{\begin{enumerate}}
\newcommand{\eenum}{\end{enumerate}}
\newcommand{\bit}{\begin{itemize}}
\newcommand{\eit}{\end{itemize}}
\newtheorem{thm}{Theorem}

\newtheorem{cor}{Corollary}

\usepackage{graphicx}

\begin{document}
\def\s{\subseteq}
\def\n{\noindent}
\def\se{\setminus}
\def\dia{\diamondsuit}
\def\la{\langle}
\def\ra{\rangle}


\title{BOUNDS OF ZAGREB INDICES AND HYPER ZAGREB INDICES}

\author{ SHAOHUI WANG$^{1,}$\footnote{Corresponding authors.  \newline Emails: S. Wang (e-mail: shaohuiwang@yahoo.com, swang@adelphi.edu), W. Gao(gaowei@ynnu.edu.cn), M.K. Jamil(m.kamran.sms@gmail.com),
M.R. Farahani (mrfarahani88@Gmail.com), J.-B. Liu(liujiabaoad@163.com).}, WEI GAO$^{2}$, MUHAMMAD K. JAMIL$^3$, \\MOHAMMAD R. FARAHANI$^4$, JIA-BAO LIU$^{5,*}$
\\
\small\emph {1. Department of Mathematics and Computer Science, Adelphi University, Garden City, NY, USA.}\\
\\\small\emph {2. School of Information Science and Technologys, Yunnan Normal University, Kunming, PR China.}\\
\\\small\emph {3. Abdus Salam School of Mathematical Sciences, Government College University, Lahore, Pakistan.}\\
\\\small\emph {4. Department of Applied Mathematics,  Iran University of Science and Technology,
Narmak, Tehran, Iran.}
\\\\
\small\emph {5. School of Mathematics and Physics, Anhui Jianzhu University, Hefei, PR China
}}
\date{}
\maketitle

\begin{abstract}
The hyper Zagreb index is a kind of extensions of Zagreb index, used for predicting physicochemical properties of organic
compounds.
Given a graph $G= (V(G), E(G))$, the first hyper-Zagreb index is the sum of the square of edge degree over edge set $E(G)$ and defined  as $HM_1(G)=\sum_{e=uv\in
E(G)}d(e)^2$, where $d(e)=d(u)+d(v)$ is the edge degree. In this work we define the second
hyper-Zagreb index on the adjacent edges  as $HM_2(G)=\sum_{e\sim f}d(e)d(f)$, where
$e\sim f$ represents the adjacent edges of $G$. By inequalities, we explore some
upper and lower bounds of these hyper-Zagreb indices, and provide the relation between Zagreb indices and hyper Zagreb indices.
\\\\ \noindent
Accepted by MATHEMATICAL REPORTS.


\end{abstract}

\section{Introduction}
\label{intro}
The graphs $G=(V(G),E(G))$ considered in this paper are finite, loopless and contain no multiple edges.
Given a graph $G= (V, E)$,  $V$ and $E$ represent the set of vertices and the set of edges with $n = |V|$ vertices and $m=|E|$ edges, respectively. For a vertex $u \in V$, the number of vertices adjacent with $u$ is called its degree $d(u)$. In a graph $G$, $\triangle$ and $\delta$ represent the maximum and the minimum degree, respectively.

In 1947, Harold Wiener introduced famous Wiener index, a most widely
known topological descriptor \cite{16}. The Winner index is the oldest and one of the most popular molecular structure
descriptors, well correlated with many physical and chemical properties of a variety of classes of
chemical compounds. Based on the success on the Wiener index, many topological indices have been introduced. Almost forty years
ago, Gutman et al. defined the important degree-based topological indices:
the first and second Zagreb indices \cite{8}. These
are defined as
$$
M_1(G)=\sum_{v\in V(G)} d(v)^2,
M_2(G)=\sum_{uv\in E(G)} d(u)d(v).
$$
In 2004, Mili$\breve{c}$evi$\acute{c}$  \cite{13} reformulated these Zagreb indices in terms of edge degrees,
$d(e)=d(u)+d(v)-2$, for $e=uv$ and defined reformulated Zagreb
indices,
$$
EM_1(G)=\sum_{e\in E(G)}d(e)^2,
EM_2(G)=\sum_{e\sim f} d(e)d(f).
$$
In 2013, Shirdel et al.  \cite{15} defined the {\it
first hyper Zagreb index} as follows,
$$HM_1(G)=\sum_{e\in E(G)}d(e)^2,$$
where $d(e)=d(u)+d(v)$. In 2016, Jamil et al.  \cite{7}
improved and extended the Shirdel's results. Based on this definition
of edge degree, we define the {\it second hyper Zagreb index} as
follows,
$$HM_2(G)=\sum_{e\sim f}d(e)d(f),$$
where
$e\sim f$ represents the adjacent edges of $G$. Furthermore, $G$ is called regular if every vertex has the same degree and edge degree regular if every edge has the same degree, respectively.

These graph invariants, based on vertex-degrees and edge-degrees of a graph, are widely used
in theoretical chemistry.
 For applications of Zagreb
indices in QSPR/QSAR and latest results, refer to  \cite{ 1, 17,  2, 5, 6, 10, 01, 02, 03, 04,  14, 18,    w2015, w2016, w1, w2, w3, w4,z1,z2}. 

As a fundamental dynamical processing system, the basis of graph structure has received considerable interest from
the scientific community. Recent work shows
that the key quantity-degree-based topological indices  to a given graph class
 on uncorrelated random scale-free networks is qualitatively reliant on
the heterogeneity of network structure. However,
in addition to the transformations of these graph basis, most real
system models (topological indices) are also characterized by degree correlations.
In this paper, we explore some properties of hyper Zagreb indices
in terms of  the number of vertices $n$, the number of edges $m$,
maximum and minimum degree $\triangle,\delta$, respectively. Also we provide the relation between hyper Zagreb  indices and
first Zagreb index $M_1(G)$.

\section{Preliminaries and main results}
After introducing the construction and structural
properties of degree-based topological indices, we will provide our main results by presenting their inequalities.
\begin{thm}
Let G be a graph with n number of vertices and m number of edges,
then
$$\delta^2\le \frac{HM_1(G)}{4m}\le \triangle^2,$$
the left and right equalities hold if and only if G is   $\delta$-regular and $\triangle$-regular, respectively.
\end{thm}
\begin{proof}
Note that 
$\delta\le d(v_i)\le \triangle, ~i=1,2,\cdots,n$. Then 
$$2\delta\le d(e_j)\le 2\triangle,~j=1,2,\cdots,m.$$
By the definition of the first hyper Zagreb index, we have
$$\delta^2\le \frac{HM_1(G)}{4m}\le \triangle^2.$$
Clearly, the equalities hold if and only if G is   $\delta$-regular and $\triangle$-regular. In particular, if $G$ is general regular connected graph, then $\delta(G) = 2$ and $\triangle(G) =n-1$.
\end{proof}

\begin{thm}
Let G be a graph with  m edges, then
$$HM_1(G)\ge \frac{M_1(G)^2}{m},$$
the equality holds if and only if G is edge degree regular.
\end{thm}
\begin{proof} Let $d(e_i)$ be the edge degree of $G$. 
By Cauchy-Schwartz inequality, we obtain
\[[d(e_1)^2+d(e_2)^2+\cdots+d(e_m)^2][1^2+1^2+\cdots+1^2]\ge [d(e_1)\cdot1+d(e_2)\cdot1+\cdots+d(e_m)\cdot1]^2.\]
Note that $\sum_{e \in E(G)} d(e) = \sum_{v \in V(G)}  d(v)^2$. By the concept of $M_1(G)$, we obtain the relation between $HM_1(G)$ and $M_1(G)$ below.
$$HM_1(G)\cdot m\ge M_1(G)^2,$$
that is,
$$HM_1(G)\ge \frac{M_1(G)^2}{m}.$$
Clearly, the equality holds if and only if every edge has the same degree, that is, $G$ is edge degree regular.
\end{proof}
\begin{thm}\label{thm3}
Let G be a graph with n vertices and m edges, then
$$HM_1(G)\le M_1(G)(m+2\delta-1)-2m(m-1)\delta,$$
the equality holds if and only if G is regular.
\end{thm}
\begin{proof} We keep the same notations as \cite{4}.
Let $d(e_i)\mu_i$ be the sum of degrees of the edges adjacent to
the edge $e_i$. We have
$$d(e_i)\mu_i=\sum_{e_i\sim e_j}d(e_j)\le \sum_{i=1}^m d(e_i)-d(e_i)-(m-1-d(e_i))2\delta.$$
Thus,
\begin{align*}
HM_1(G)&=\sum_{e_i\in E(G)}d(e_i)^2=\sum_{i=1}^m d(e_i)\mu_i\\
&\le \sum_{i=1}^m[\sum_{i=1}^m d(e_i)-d(e_i)-(m-d(e_i))2\delta]\\
&= M_1(G)(m+2\delta-1)-2m(m-1)\delta.
\end{align*}
Clearly, the equality holds if and only if $G$ is regular.
\end{proof}

By the results of \cite{9} that $M_1(G)\le
2m(\triangle+\delta)-n\triangle \delta$, where the equality holds if and
only if $G$ is regular,  we have the following corollary.
\begin{cor}Let G be a graph with n vertices and m edges, then
$$HM_1(G)\le (2m(\triangle+\delta)-n\triangle \delta)(m+2\delta-1)-2m(m-1)(\delta-1),$$
where the equality holds if and only if G is regular.
\end{cor}

\begin{thm}\label{a}
Let G be a graph with n vertices, m edges and minimum degree
$\delta\ge 2$, then
$$HM_1(G)\le\frac{(\triangle+\delta)^2}{4m\triangle\delta}M_1(G)^2,$$
the equality holds if and only if G is a regular graph, or there are
exactly $\frac{m\delta}{\triangle+\delta}$ edges of degree
$2\triangle$ and $\frac{m\triangle}{\triangle+\delta}$ edges of
degree $2\delta$ such that $(\triangle+\delta)$ divides $m\delta$.
\end{thm}

\begin{proof}
If $a,a_1,a_2,\cdots,a_m$ and $b,b_1,b_2,\cdots,b_m$ are positive
real numbers such that $a\le a_i\le A$, $b\le b_i\le B$ for $1\le
i\le m$ with $a<A$ and $b<B$, by P$\acute{o}$lya-Szeg$\acute{o}$ Inequality\cite{20}, we have
$$\sum_{i=1}^m {a_i}^2\cdot \sum_{i=1}^m {b_i}^2\le \frac{1}{4}\Big(\sqrt{\frac{AB}{ab}}+\sqrt{\frac{ab}{AB}}\Big)^2\cdot\Big(\sum_{i=1}^m {a_ib_i}\Big)^2,$$
and the equality holds if and only if the numbers
$$k=\frac{\frac{A}{a}}{\frac{A}{a}+\frac{B}{b}},l=\frac{\frac{B}{b}}{\frac{A}{a}+\frac{B}{b}}$$
are integers, $a=a_1=a_2=\cdots=a_k$;
$A=a_{k+1}=a_{k+2}=\cdots=a_m$ and $B=b_1=b_2=\cdots = b_l$;
$b=b_{l+1}=b_{l+2}=\cdots=b_m$. If we allow $a=A$ or $b=B$, the
equality holds if $AB=ab$, i.e., $A=a=a_1=a_2=\cdots=a_m$ and
$B=b=b_1=b_2=\cdots,b_m$. By setting the values $a_i=1$ and
$b_i=d(e_i)$ for $i=1,2,\cdots,m$, we obtain
$$\sum_{i=1}^m {1}^2\cdot \sum_{i=1}^m {d(e_i)}^2\le \frac{(AB+ab)^2}{4ABab}\cdot\Big(\sum_{i=1}^m {d(e_i)}\Big)^2.$$
So,
$$mHM_1(G)\le \frac{(AB+ab)^2}{4ABab}\cdot M_1(G)^2.$$
Now since $a\le a_i\le A$, we have $a=A=1$ and since $b\le b_i\le
B$, we have $b=2\delta$ and $B=2\triangle$. Hence,
$$HM_1(G)\le\frac{(2\triangle+2\delta)^2}{16\triangle\delta}M_1(G)^2,$$
which is the expected result. In the last expression, the equality
holds if and only if G is a regular graph, or there are exactly
$\frac{m\delta}{\triangle+\delta}$ edges of degree $2\triangle$
and $\frac{m\triangle}{\triangle+\delta}$ edges of degree
$2\delta$ such that $(\triangle+\delta)$ divides $m\delta$.
\end{proof}

\begin{cor}\label{cor2}
Let G be a graph with n vertices, m edges and minimum degree
$\delta\ge 2$, then
$$HM_1(G)\le \frac{(n+1)^2}{8m(n-1)}M_1(G)^2,$$
the equality holds if G has exactly $\frac{m}{n-1}$ edges of degree 2(n-2) and $\frac{m(n-2)}{n-1}$ edges of degree 2 such that n-1 divides m.
\end{cor}

\begin{proof}
Note that 
$$\frac{(\triangle+\delta)^2}{\triangle\delta}=\frac{\triangle}{\delta}+\frac{\delta}{\triangle}+2.$$
By Theorem \ref{a}, we have
$$HM_1(G)\le\Big[\frac{\triangle}{\delta}+\frac{\delta}{\triangle}+2\Big]M_1(G)^2.$$
As the function $f(x)=x+\frac{1}{x}$ is increasing for $x\ge1$, so $\Big[\frac{\triangle}{\delta}+\frac{\delta}{\triangle}+2\Big]$ is increasing for $\frac{\triangle}{\delta}\ge1$. Now for $\delta\ge2$, $1\le \frac{\triangle}{\delta}\le \frac{n-1}{2}$. So, $\Big[\frac{\triangle}{\delta}+\frac{\delta}{\triangle}+2\Big]\le \frac{(n+1)^2}{2(n-1)}$. So,
$$HM_1(G)\le \frac{(n+1)^2}{8m(n-1)}M_1(G)^2,$$
the equality holds if $G$ has exactly $\frac{m}{n-1}$ edges of degree 2(n-2) and $\frac{m(n-2)}{n-1}$ edges of degree $2$ such that $n-1$ divides $m$.
\end{proof}

\begin{cor}
Let G be a graph with n vertices and m edges, then
$$HM_1(G)\le \frac{m^3(n+1)^6}{16n^2(n-1)^2},$$
the equality holds if and only if $G\cong K_3$.
\end{cor}
\begin{proof}
Note that  \cite{9}  $M_1(G)\le
\frac{m^2(n+1)^2}{2n(n-1)}$, for $\delta \ge 2$ with the equality holds if
and only if $G\cong K_3$. Thus, Corollary \ref{cor2} yields the
result.
\end{proof}

\begin{thm}\label{thm5}
Let G be a graph with n vertices and m edges, then
$$HM_1(G)\le 2(\triangle+\delta)M_1(G)-4m\triangle\delta, $$
the equality holds if and only if G is a regular graph.
\end{thm}

\begin{proof}
Suppose $a_i$, $b_i$, $p$ and $P$ are real numbers such that
$pa_i\le b_i\le Pa_i$ for $i=1,2,\cdots,m$, then we have
Diaz-Metcalf inequality\cite{19},
$$\sum_{i=1}^mb_i^2+pP\sum_{i=1}^ma_i^2\le (p+P)\sum_{i=1}^ma_ib_i,$$
and the equality holds if and only if $b_i=pa_i$ or $b_i=Pa_i$ for
every $i=1,2,\cdots,m$. By setting $a_i=1$ and $b_i=d(e_i)$, for
$i=1,2,\cdots,m$, from the above inequality we obtain
$$\sum_{i=1}^md(e_i)^2+2\triangle\cdot 2\delta\sum_{i=1}^m1^2\le2(\triangle+\delta)\sum_{i=1}^md(e_i).$$
and $$HM_1(G)\le 2(\triangle+\delta)M_1(G)-4m\triangle\delta.$$
Thus, the equality holds if and only if $G$ is a regular graph.
\end{proof}

By the results of \cite{9} we have, $M_1(G)\le
2m(\triangle+\delta)-n\triangle\delta$, with the equality holds if and only
if $G$ is regular. So, we have the following result
\begin{cor}
Let G be a graph with n vertices and m edges, then
$$HM_1(G)\le 4m(\triangle+\delta)^2-\triangle\delta(n+4m).$$
\end{cor}

\begin{thm}
Let G be a graph with n vertices and m edges, then
$$\delta^2\le \frac{HM_2(G)}{2(M_1(G)-2m)}\le \triangle^2,$$
the equality holds if and only G is a regular graph.
\end{thm}

\begin{proof}
The number of pairs of edges which have a common end point is
$\sum_{i=1}^n \left( \begin{array}{c} d_i \\ 2 \end{array} \right) =\frac{1}{2}M_1(G)-2m$. Also, $2\delta\le d(e_j)\le
2\triangle$,  for $j=1,2,\cdots,m$. So, from the definition of
second hyper Zagreb index, we have
$$4\Big(\frac{1}{2}M_1(G)-m\Big)\delta^2\le HM_2(G)\le 4\Big(\frac{1}{2}M_1(G)-m\Big)\triangle^2,$$
and $$\delta^2\le \frac{HM_2(G)}{2(M_1(G)-2m)}\le \triangle^2,$$
the equality holds if and only if $G$ is a regular graph.
\end{proof}

\begin{thm}
Let G be a graph with n vertices and m edges, then
$$HM_2(G)\ge \frac{M_1(G)^3}{2m^2},$$
the equality holds if and only if G is regular.
\end{thm}

\begin{proof}
For arithmetic and geometric mean inequality,
$$\frac{1}{N}\sum_{e_i\sim e_j}d(e_i)d(e_j)\ge \Big[\prod_{e_i\sim e_j}d(e_i)d(e_j)\Big]^{\frac{1}{N}}=\Big[\prod_{i=1}^md(e_i)^{d(e_i)}\Big]^{\frac{1}{N}},$$
where $N=\frac{1}{2}M_1(G)$. Suppose that
$L=\prod_{i=1}^md(e_i)^{d(e_i)}$. Taking natural logarithm on both
sides, we obtain
$$ln\,\, L=\sum_{i-1}^md(e_i)ln\,\, d(e_i)\ge \sum_{i=1}^md(e_i)ln\,\, \frac{1}{m}\sum_{i=1}^md(e_i),$$
and $$L\ge \Big(\frac{M_1(G)}{m}\Big)^{M_1(G)}.$$
Hence,
\begin{align*}
HM_2(G)&\ge N\Big[\frac{M_1(G)}{m}\Big]^{\frac{M_1(G)}{N}}\\
&= \frac{M_1(G)^3}{2m^2}.
\end{align*}
Clearly, the equality holds if and only if $G$ is a regular graph.
\end{proof}

\begin{thm}\label{thm8}
Let G be a graph with n vertices and m edges, then
$$HM_2(G)\le \frac{1}{2}M_1(G)^2-\delta(m-1)M_1(G)+(\delta-\frac{1}{2})HM_1(G),$$
the equality holds if and only if G is regular.
\end{thm}
\begin{proof}
By the result of Theorem 3, we have
$$d(e_i)\mu_i=\sum_{e_i\sim e_j}d(e_j)\le \sum_{i=1}^m d(e_i)-d(e_i)-(m-1-d(e_i))2\delta.$$
Thus
\begin{align*}
HM_2(G)&=\sum_{e_i\sim e_j}d(e_i)d(e_j)=\frac{1}{2}\sum_{i=1}^md(e_i)^2\mu_i=\frac{1}{2}\sum_{i=1}^md(e_i)\Big(\sum_{e_i\sim e_j}d(e_j)\Big)\\
&\le\frac{1}{2}\sum_{i=1}^md(e_i)\Big(\sum_{i=1}^md(e_i)-d(e_i)-2(m-1-d(e_i)\delta)\Big)\\
&=\frac{1}{2}M_1(G)^2-\frac{1}{2}HM_1(G)-\delta(m-1)M_1(G)+\delta
HM_1(G).
\end{align*}
The expected result is obtained from the above proof process.

Clearly, the equality holds when the graph $G$ is regular.
\end{proof}

\begin{cor}
Let G be a graph with n vertices, m edges and $\delta$ minimum
degree, then
$$HM_2(G)\le \frac{1}{2}K^2-\Big(\delta(m-1)+(\delta-\frac{1}{2})(m+2\delta-1)\Big)K-m(m-1)(2\delta-1)\delta,$$
where $K=M_1(G)$ or $K=2m(\triangle+\delta-1)-n\triangle\delta$
with the equality if and only if G is regular.
\end{cor}

\begin{proof}
Using Theorem \ref{thm3} and Theorem \ref{thm8}, we obtain the
expected result with $K=M_1(G)$. Moreover,  we have $M_1(G)\le
2m(\triangle+\delta-1)+2m-n\triangle\delta$ \cite{9}
with the equality holds if and only if $G$ is regular, so the expected result
clearly follows for $K=2m(\triangle+\delta-1)-n\triangle\delta$.
\end{proof}

\noindent{\bf Acknowledgments}\\
 The authors would like to express their sincere gratitude to   the  anonymous
referees and the editor for many friendly and helpful suggestions, which led to great deal
of improvement of the original manuscript.

This work is partially supported by National Natural Science Foundation of China (nos. 11601006, 11471016,
11401004, Anhui Provincial Natural Science Foundation (nos. KJ2015A331,
KJ2013B105).

\noindent
\small\emph {1. Department of Mathematics and Computer Science, Adelphi University, Garden City, NY, USA.}\\
\\\small\emph {2. School of Information Science and Technologys, Yunnan Normal University, Kunming, PR China.}\\
\\\small\emph {3. Abdus Salam School of Mathematical Sciences, Government College University, Lahore, Pakistan.}\\
\\\small\emph {4. Department of Applied Mathematics,  Iran University of Science and Technology,
Narmak, Tehran, Iran.}
\\\\
\small\emph {5. School of Mathematics and Physics, Anhui Jianzhu University, Hefei, PR China
}

\end{document}